\documentclass[11pt]{amsart}
\usepackage{color}
\usepackage{amsmath}
\usepackage{mathtools}
\usepackage{amssymb}

\newcommand{\bt}{\begin{Theorem}}
\newcommand{\et}{\end{Theorem}}
\newcommand{\bi}{\begin{itemize}}
\newcommand{\ei}{\end{itemize}}
\newcommand{\bea}{\begin{eqnarray}}
\newcommand{\ba}{\begin{array}}
\newcommand{\eea}{\end{eqnarray}}
\newcommand{\ea}{\end{array}}

\newtheorem{Definition}{Definition}[section]
\newtheorem{theorem}[Definition]{Theorem}
\newtheorem{Lemma}[Definition]{Lemma}

\newtheorem*{theoremA*}{Theorem A}
\newtheorem*{theoremB*}{Theorem B}
\newtheorem*{Proofofmainthm*}{Proof of main theorem}
\newcommand{\be}{\begin{equation}}
\newcommand{\ee}{\end{equation}}
\newcommand{\bvb}{\begin{verbatim}}

\newcommand{\evb}{\end{verbatim}}

\newcommand{\R}{\mathbb R}%
\newcommand{\C}{\mathbb C}%
\newcommand{\X}{\mathbb X}%
\allowdisplaybreaks[1]

\begin{document}
\baselineskip16pt

\author[Pratyoosh Kumar and Mannali Sajjan]{Pratyoosh Kumar and Manali Sajjan }
\address{Department of Mathematics, Indian Institute of Technology Guwahati, 781039, India.
E-mail: pratyoosh@iitg.ac.in and smanali@iitg.ac.in}

\title[Regularity of Solution of the Schr\"odinger Equation ]
{Regularity of Solution of the Schr\"odinger Equation on Symmetric Space}
\thanks{Second Author is supported by Institute fellowships of IIT Guwahati.}

\begin{abstract} 
In this article, we investigate the behavior of solutions \( u(x,t) \) to the fractional Schrödinger equation on rank symmetric spaces of non-compact type. We proved that as time \( t \) approaches $0$, then $u(x,t)$ converges pointwise almost everywhere to the initial radial data \( f \), provided that \( f \in H^s(\mathbb{X}) \) with \( s > \frac{1}{2} \). This result extends Sjölin's results in this setting.

\end{abstract}

\maketitle

\section{Introduction}
Let $u(x,t)$ be the solution of the Schr\"odinger equation
\begin{equation}
\left\{\begin{array}{l}
\displaystyle i\frac{\partial u(x,t)}{\partial t} = \frac{\partial^2 u(x,t)}{\partial x^2}, \qquad \;(t,x)\in \mathbb{R}\times {\mathbb{R}}^n, \\ [12pt]
u(x,0)=f(x).
\end{array}\right.  \label{1}
\end{equation}
By taking the Fourier transform on both sides with respect to $x$-variable we have 
\[u(x,t)=S_tf(x)=\int_{0}^{\infty} e^{ix\xi}e^{it{\xi}^2}\widehat{f}(\xi)\; d\xi.\] 
One can view $S_t$ as a family of linear operators on $\strut L^2({\mathbb{R}}^n)$ such that \[\widehat{S_t f}(\xi)=e^{it|\xi|^2}\hat{f}(\xi),\] 
and by Plancheral theorem 
\[\displaystyle\lim_{t\rightarrow 0} S_tf=f,\] 
in $\strut L^2$-norm. In \cite{C}, Carleson posed  a question regarding the amount of regularity required to the initial data $f$ such that $S_tf(x) \rightarrow f \;a.e.\; x \in \mathbb{R}^n$ as $t$ goes to \(0\). When \(n=1\), Carleson proved that if $f \in H^s(\mathbb{R})$, $s\geq \frac14$ and support of $f$ is compact, then $S_tf(x) \rightarrow f\; a.e.\; x \in \mathbb{R}^n$, as $t$ goes to $0$. In \cite{DK}, Dahlberg and Kening proved that the condition $s\geq \frac14$ is sharp. This result has been further improved by many authors, including Sj\"olin, Bourgain, Moyua, Vargas, Vega, and Lee, in higher dimensions (see \cite{BO, DU, Lee, Mo} for details), leading to a sharp result. The similar results were given  in \cite{WZ}, where Xing Wang and Chunjie Zhang addressed the same problem in a manifold setting and obtained results analogous to those in Euclidean space.

 In this paper, we are interested in the fractional case of this problem in rank-one symmetric space of non-compact type with radial initial data. Our result is an extension of Sj\"olin's result on $\R^n$. 
 
 In \cite{Sj}, Sj\"olin proved that for $n>2$, if  $f \in H^s({\mathbb{R}}^n)$, $s> \frac{1}{2}$ then $S_tf(x) \rightarrow f\; a.e.\; x\in \mathbb{R}^n$, as $t$ goes to $0,$ where $u(x,t)=S_tf(x)$ is the solution of the  fractional Schr\"odinger equation 
\begin{equation}
\left\{\begin{array}{l}
\displaystyle i\frac{\partial u(x,t)}{\partial t} = {\Delta}^{\frac{a}{2}} u(x,t), \qquad (t,x)\in \mathbb{R}\times {\mathbb{R}}^n,a>1 , \\ [12pt]
u(x,0)=f(x).
\end{array}\right. \label{2}
\end{equation}
 
Here we also generalize the definition of  $S_tf(x)=\displaystyle\int_{0}^{\infty} e^{ix\xi}e^{it|\xi|^a}\widehat{f}(\xi) d\xi$. Notably, when $a=2,$ this reduces to the usual Schr\"odinger equation discussed earlier. We define 
\be S^*f(x)=\displaystyle\sup_{0<t<1} |S_tf(x)|,\label{maximaldefinition}\ee and in \cite{Sj}, Sj\"olin proved the following result:
\begin{theorem}\label{sj}
	If $n\geq 3$ and $a>1,$ then $\|S^*f\|_{\strut L^2(B)} \leq c \|f\|_{\strut H^s(\mathbb{R}^n)}$ holds for $s>1/2$.
	\end{theorem}
In \cite{WZ} Wang and Zhang addressed the same problem for real hyperbolic space. They proved the similar bound like Theorem \ref{sj}, for $a=2$. 

In this present work, we proved the analogous result of Sj\"olin. Our result also extends the result of Wang and Zhang in for all rank one symmetric space  and for all $a>1$.
Now we formulate the problem in symmetric space. Let  $\mathbb{X}$ be a symmetric space and let 
$\Delta$ be the Laplace-Beltrami operator in $\mathbb{X}$. The fractional Schr\"odinger equation in  $\mathbb{X}$ is given by
 \begin{equation}
\left\{\begin{array}{l}
\displaystyle i\frac{\partial u(x,t)}{\partial t} = {\Delta}^{\frac{a}{2}} u(x,t), \qquad (t,x)\in \mathbb{R}\times \mathbb{X},~a>1, \\ [12pt]
u(x,0)=f(x), \qquad f ~\text{radial}.
\end{array}\right. \label{3}
\end{equation}
By taking Harish-Chandra transform in $x$-variable the solution $u(x,t)$ of equation \eqref{3} is given by $$u(x,t)=S_tf(x)=c \int_{0}^{\infty} \hat{f}(\lambda)e^{it(|\lambda|^2+|\rho|^2)^{\frac{a}{2}}} \phi_{\lambda}(x) |c(\lambda)|^{-2}d\lambda.$$
The meaning of these symbols will be explained later. Now we state the main result of this paper:
\begin{theoremA*}
	The solution $u(x,t)$ of the equation \eqref{3} converges pointwise to the initial data $f$, as $t$ goes to $0$, whenever $f \in  H^s(\mathbb{X})$, $s> \frac{1}{2} $, and also support of $f$ is compact. \label{TheoremA}
	\end{theoremA*}
The standard method for proving pointwise convergence involves obtaining an estimate for the corresponding maximal operator. We will follow this approach as well. Our proof will align with Sj\"olin's ideas; to establish some mixed norm estimates and by using interpolation and inclusion results from Sobolev spaces, derive the required estimate for the corresponding maximal operator. Our proof of Theorem A is mainly based on the following maximal estimate:

 \begin{theoremB*}\label{Theorem B}
     If $a>1$, and $B$ is ball of any arbitrary radius in $\mathbb{X}$, then 
     \be \|S^*f\|_{\strut L^2(B)} \leq c \|f\|_{\strut H^s(\mathbb{X})},\label{maximalestimate}\ee holds for $s>1/2$, Where $S^*f$ is same as previously defined in \eqref{maximaldefinition}  and $f \in C_c^{\infty}(K \backslash  G/K).$
 \end{theoremB*}
\section{Notation and Preliminaries}
In this section, we will discuss the basic facts of symmetric spaces and their spherical functions.  
Let $G$ be a connected 
noncompact semisimple Lie group with finite center, and $\mathfrak{g}$ be the Lie algebra of $G$. 
Let $\theta$ be a Cartan involution of $\mathfrak{g}$, and $\mathfrak{g}=\mathfrak{k}\oplus\mathfrak{p}$ be the 
associated Cartan decomposition. Let $K=\exp \;\mathfrak{k}$ be a maximal compact subgroup of $G$ and let $X=G/K$ 
be the associated Riemannian symmetric space. If $o=eK$ denotes the identity coset, then for $g\in G$ the quantity $r(g)$ denotes 
the Riemannian distance of the element $g.o$ from the identity element. Let $\mathfrak{a}$ be a maximal abelian subspace of 
$\mathfrak{p}$, $A=\exp \;\mathfrak{a}$ be the corresponding subgroup of $G,$ and $M$ the centralizer of $A$ in $K$. 

From this point onward, we will assume that the group $G$
has real rank one, which implies that dim $\mathfrak{a}=1$. Let $a_{\mathbb{R}}^*$ be the real dual of $a$ and for $\alpha \in a_{\mathbb{R}}^*,$ we define ${\mathfrak{g}}_{\alpha} = \{X \in \mathfrak{g} = [H,X] = \alpha(H) X$  $\text{for all} ~ H \in \mathfrak{a} \}$. We say $\alpha$ is nonzero if $\text{dim}_{\strut \mathbb{R}}{\mathfrak{g}}_{\alpha} > 0.$  In this case, it is well known that the set of nonzero roots is either of the 
form $\{-\alpha,\alpha\}$ or $\{-\alpha,-2\alpha,\alpha,2\alpha\}$. 

Let $\mathfrak{n}= \mathfrak{g}_\alpha\oplus\mathfrak{g}_{2\alpha}$ and $N=$exp $\mathfrak{n}$. Let $H_0$ be the unique element of $\mathfrak{a}$ such that $\alpha(H_0)=1$ and $A=\{a_s : a_s=\exp \;sH_0, s\in\R\}$. We identify $\mathfrak{a^*}$ (the dual of $\mathfrak{a}$) and $\mathfrak{a^*_{\C}}$ (the complex dual of $\mathfrak{a}$) by $\R$ and $\C$
via the identification $t\mapsto t\alpha$ and $z\mapsto z\alpha$, $t\in \R$ and $z\in \C$ respectively.

Let 
$m_1=\text{dim}\;\mathfrak{g}_\alpha$, $m_2=\text{dim}\;\mathfrak{g}_{2\alpha}$ and $\rho=\frac{1}{2}(m_1+2m_2)\alpha$ be the
half sum of positive roots. We will denote dim$\mathbb{X}$ by $n$,\;$n=m_1+m_2+1$. By abuse of notation we will denote $\rho(H_0)= \frac{1}{2}(m_1+2m_2)$ by $\rho$.

The functions defined on \( \mathbb{X} \) can be viewed as right \( K \)-invariant functions on \( G \). Furthermore, the radial functions on \( \mathbb{X} \) are \( K \)-biinvariant functions on \( G \).
Let $G=K\overline{A^+}K$ be the Cartan decomposition $G,$ where $\overline {A^+}=\{a_t\in A  : t\geq 0\}$. The Haar measure related to the Cartan decomposition is given by:
\be \int_G f(g)\; dg =C \int_K\int_0^\infty \int_K f(k_1a_tk_2)D(t)\;dk\;dt\;dk,\ee\label{cartan}
where $D(t)=(\sinh t)^{m_1}(\sinh 2t)^{m_2}.$

Let $\Delta$ denote the Laplace-Beltrami operator. The Spherical function $\phi_{\lambda}$ on $\mathbb{X}$ are radial eigenfunction of $\Delta$ such that $$\Delta \phi_{\lambda}=-({\lambda}^2+{\rho}^2)\phi_{\lambda} ,\;\;  \phi_{\lambda}(0)=1,\;\; \lambda \in \mathbb{C}.$$
In the Iwasawa decomposition of $G=KAN$, every $g\in G$ can be uniquely written as 
\be g=k(g)\;\exp H(g)\;n(g).\label{iwasawa}\ee
The integral expression of $\phi_{\lambda}$ is given by
$$\phi_{\lambda}(x)=\int_{K}e^{-(i\lambda+\rho)H(x^{-1}k)}dk.$$

For $\lambda \in \R$, the spherical Fourier transform of a suitable radial function $f$ on $\X$ is defined by

\be\widehat{f}(\lambda)=\int_{\X} f(x)\phi_{\lambda}(x) dx.\ee

We also need the following important formulas:
\be f(x)=c\int_{0}^{\infty} \widehat{f}(\lambda) \phi_{\lambda}(x)|c(\lambda)|^{-2} d\lambda. \;\;\;\;(\textbf{Inversion formula}) \ee

\be \|f\|^2_{L^2(X)}=\int_{0}^{\infty} |\widehat{f}(\lambda)|^2|c(\lambda)|^{-2}d\lambda, \quad \text{for}\; f \in C_c^{\infty}(K \backslash  G/K).\;\; (\textbf{Plancherel formula})\ee

The function $c(\lambda)$, the famous Harish-Chnadra $\bold c$-function. We also need the following well-known estimate of the Plancherel density, (see \cite{An}):
\be |c(\lambda)|^{-2} \leq |\lambda|^2 (1+|\lambda|)^{(n-3)} ,\;\; \forall \lambda \in \mathbb{R}.\label{clambda}\ee

\subsection{Some estimates of $\phi_{\lambda}$}
We need some estimates of the spherical function, which is crucial for the proof. These estimates we will derived from the series expansion of $\phi_\lambda$. Most of these estimates are given in \cite{Io, st}.

 It is follow from \cite[Theorem 2.1]{st}, that there exist $R_0,R_1(>1)$ such that for any $t$ with $0\leq t \leq R_0$ and any $M \geq 0$,
\be \phi_{\lambda}(\exp tH_0)=c_0\left[\frac{t^{n-1}}{D(t)}\right]^{\frac{1}{2}} \displaystyle\sum_{0}^{\infty} t^{2m} a_m(t) \mathcal{J}_{\frac{n-2}{2} +m}(\lambda t),\ee\label{spherical}

where,\; $\mathcal{J}_{\mu}(z)=\displaystyle \frac{J_{\mu}(z)}{z^{\mu}} \Gamma\left(\mu + \frac{1}{2}\right) \Gamma\left(\frac{1}{2}\right) 2^{\mu -1}$, and ${J_{\mu}(z)}$ is the standard Bessel function and $D(t)=(sht)^{m_1} (sh2t)^{m_2} .$ The above expression of $\phi_\lambda$
can also be written as:
\be \phi_{\lambda}(\exp tH_0)=c_0 \left[\frac{t^{n-1}}{D(t)}\right]^{\frac{1}{2}} \displaystyle\sum_{0}^{M} t^{2m} a_m(t) \mathcal{J}_{\frac{n-2}{2} +m}(\lambda t)+E_{M+1}(\lambda t),\label{spherical 1}\ee
where, $$a_0(t)=1 , \text{and}\;\; |a_m(t)|\leq cR_1^{-m}.$$

\begin{equation}
|E_{M+1}(\lambda t)|\leq
\left\{\begin{array}{l}
c_Mt^{2(M+1)}, \;\;\;\;\;\;\;\;\;\;\;\;\;\;\;\;\;\;\;\;\;\;\; \text{if}\;\;\; |\lambda t| \leq 1,\\ [8pt]
c_M t^{2(M+1) (\lambda t)^{-(\frac{n-1}{2} +M+1)}},\; \text{if}\;\;\; |\lambda t|>1.
\end{array}\right.\label{localestimate}
\end{equation}
The Harish-Chandra series for the Spherical function is given by:
\be \phi_{\lambda}(\exp \;tH_0)=c(\lambda)e^{(i\lambda - \rho)t}\varphi_{\lambda}(t)+c(-\lambda)e^{-(i\lambda - \rho)t} \varphi_{-\lambda}(t), \label{Harish-Chandraseries}\ee
where $\varphi_{\lambda}(t)=\sum_{n=0}^{\infty} \Gamma_n(\lambda)e^{-nt}$ and $c(\lambda)$ is the Harish-Chandra $\bold{c}$-function, which is given by,
$$c(\lambda)=\Gamma\left(\frac{m_1}{2}\right)\Gamma\left(\frac{m_2}{2}\right)\frac{\Gamma(i\lambda) \Gamma(\frac{m_1+i\lambda}{2})}{\Gamma(\frac{\rho}{2} + i\lambda) \Gamma(\frac{\rho+i\lambda}{2})}.$$
Using the notations used in \cite{Io} , we can write 
\[\phi_{\lambda}(t)=e^{-\rho t}(e^{i\lambda t} c(\lambda) a_2 (\lambda,t) + e^{-i \lambda t} c(- \lambda) a_2(-\lambda,t)).\]
Where the functions $a_2$ satisfies the inequality
\begin{equation}
\left|\frac{{\partial}^{\alpha}}{\partial {\lambda}^{\alpha}} \frac{{\partial}^l}{\partial t^l} a_2(\lambda,t)\right| \leq C(1+|\lambda|)^{-\alpha},\label{a_2estimate}
\end{equation}
for all integers $\alpha \in [0,N]$ and for all $t \geq 1$ and $\lambda \in \mathbb{R}$.
\begin{Lemma} Let $\lambda \in \R,$ the elementary spherical function $\phi_\lambda$ satifies the following properties:
\bi
\item[1.] $\phi_\lambda= \phi_{-\lambda},$
\item[2.] $|\phi_\lambda (x)|\leq 1,$ for all $x\in \X.$
 
\ei
\end{Lemma}

We conclude this section by recalling the defintion of \textbf{Sobolev-spaces} given in \cite{An}. For $s \in \mathbb{R},$ we define $\strut H^s(\mathbb{X})$ as the image of $\strut L^2(\mathbb{X})$ under $(- \Delta)^{-\frac{s}{2}}$, equipped with the norm
$$\|f\|_{\strut H^s}= \|(- \Delta)^{\frac{s}{2}} f\|_{\strut L^2}=\int_{0}^{\infty} ({\lambda}^2 + {\rho}^2)^s |\hat{f}(\lambda)|^2 d\lambda.$$

\section{Proof of Theorem A and Theorem B}

In this section, we first  prove the Theorem B. We choose $\alpha_0 \in C_c^{\infty}(\mathbb{X})$, a radial function  and $\psi_0 \in C_c^{\infty}(\mathbb{R})$, an even function  such that both are real cut-off functions, and $\alpha_0 \equiv 1$ in a ball $B$ contained in the support of $\alpha_0,$ $\psi_0 \equiv 1\; \text{in}\; [0,1].$ Now set 
\begin{equation}\label{Sfdefinition}
 Sf(x,t)  =\alpha_0(x)\psi_0(t) S_tf(x) = c\;\alpha_0(x)\psi_0(t) \int_{0}^{\infty} \widehat{f}(\lambda) e^{it(|\lambda|^2+|\rho|^2)^{\frac{a}{2}}} \phi_{\lambda}(x) |c(\lambda)|^{-2}d\lambda. 
\end{equation}
We define \be\|Sf\|_{\strut L^2(\strut H_s)} = \left(\int_{\mathbb{X}} \|Sf(x,.)\|_{\strut H_s}^2 dx\right)^{\frac{1}{2}}.\label{mixednormdef}\ee
Our plan is as follows: By Plancheral theorem, we have
$$    \|Sf\|_{\strut L^2(\strut H^0)} \leq \|f\|_{\strut H^{-s}(\mathbb{X})},\;\;\;\; \text{and}\;\;\; 
\|Sf\|_{\strut L^2(\strut H^1)} \leq \|f\|_{\strut H^{-s+a}(\mathbb{X})}.
$$
Interpolating between these two, yields
$$\|Sf\|_{\strut L^2(\strut H^r)} \leq \|f\|_{\strut H^{-s+ra}(\mathbb{X})},~0\leq r \leq 1,$$
and this, together with Sobolev imbedding further gives:
$$\|\sup_{0\leq t\leq 1}S_tf(x) \|_{\strut L^2(\mathbb{X})} \leq \|f\|_{\strut H^{s}(\mathbb{X})},~s > \frac{1}{2}.$$

First we shall prove that
\be \|Sf\|_{\strut L^2(\mathbb{X} \times \mathbb{R})} \leq C\|f\|_{{\mathbb{H}}^{-s}(\mathbb{X})}\; ,\; s=\frac{a-1}{2}. \label{firstL2estimate}\ee
 We have, 
  
 \begin{align*}
  &\int_{X}\int_{\mathbb{R}} |Sf(x,t)|^2 dx dt  \nonumber \\ 
  & =c \int_{\mathbb{X}}\int_{\mathbb{R}} \alpha_0(x)^2 \psi_0(t)^2 \left(\int_{0}^{\infty} \widehat{f}(\lambda) e^{it({\lambda}^2 +{\rho}^2)^{\frac{a}{2}}} \phi_{\lambda}(x) |c(\lambda)|^{-2} d\lambda \right) \Big(\int_{0}^{\infty} \bar{\widehat{f}(\eta)} e^{-it({\eta}^2 +{\rho}^2)^{\frac{a}{2}}} \phi_{\eta}(x) |c(\eta)|^{-2} d\eta \Big) dx dt \nonumber \\ 
 & =c\int_{0}^{\infty} \int_{0}^{\infty}\left( \int_{\mathbb{X}} \alpha_0(x)^2\phi_{\lambda}(x)\phi_{\eta}(x) dx \right) \left(\int_{\mathbb{R}}\psi (t)e^{-it(({\eta}^2 +{\rho}^2)^{\frac{a}{2}}-({\lambda}^2 +{\rho}^2)^{\frac{a}{2}})}  dt\right) \widehat{f}(\lambda) \bar{\widehat{f}(\eta)}|c(\lambda)|^{-2} |c(\eta)|^{-2} d\lambda d\eta \nonumber\\  
& =c \int_{0}^{\infty} \int_{0}^{\infty} \left( \int_{\mathbb{X}} \alpha(x)\phi_{\lambda}(x)\phi_{\eta}(x) dx \right) \hat{\psi}(({\eta}^2 +{\rho}^2)^{\frac{a}{2}}-({\lambda}^2 +{\rho}^2)^{\frac{a}{2}})) \widehat{f}(\lambda) \bar{\widehat{f}(\eta)}|c(\lambda)|^{-2} |c(\eta)|^{-2} d\lambda d\eta, 
\end{align*}
In the above expression, we put  $\alpha={\alpha_0}^2 $ , $\psi={\psi_0}^2.$  
For $\lambda,\eta \in (0,\infty),$ we define:
\be K(\lambda,\eta)= ({\rho}^2+|\lambda|)^s  ({\rho}^2+|\eta|)^s \Big ( \int_{\mathbb{X}} \alpha(x)\phi_{\lambda}(x)\phi_{\eta}(x) dx \Big)\hat{\psi}(({\eta}^2 +{\rho}^2)^{\frac{a}{2}}-({\lambda}^2 +{\rho}^2)^{\frac{a}{2}})).\label{k-function}\ee
Define,
$$Tf(\lambda)=\int_{\mathbb{R}} K(\lambda,\eta) f(\eta) |c(\eta)|^{-2} d\eta.$$
If we set $h(\lambda)=\hat{f}(\lambda)(|\rho|^2 + |\lambda|)^{-s}$, then
 \begin{align*} 	
 	\int_{X}\int_{\mathbb{R}} |Sf(x,t)|^2 dx dt & = \int_{0}^{\infty} \int_{0}^{\infty} K(\lambda,\eta) (|\rho|^2 + |\lambda|)^{-s} (|\rho|^2 + |\eta|)^{-s} \hat{f}(\lambda) \bar{\widehat{f}(\eta)} |c(\lambda)|^{-2} |c(\eta)|^{-2} d\lambda d\eta \nonumber \\
 	& = \int_{0}^{\infty} \Big( \int_{0}^{\infty} K(\lambda,\eta) \bar{h(\eta)} |c(\eta)|^{-2} d\eta \Big) h(\lambda) |c(\lambda)|^{-2} d\lambda \nonumber\\
 	& = \int_{0}^{\infty} T(\bar{h(\lambda)}) h(\lambda) |c(\lambda)|^{-2} d\lambda.	
 \end{align*}
If $T$ is bounded operator on $\strut L^2 (\mathbb{R},|c(\lambda)|^{-2} d\lambda)$, then by H\"olders inequality we have,
$$\int_{X}\int_{\mathbb{R}} |Sf(x,t)|^2 dx dt \leq \|T\bar{h}\|_2 \|h\|_2 \leq  C\|h\|_2 ^2 \leq \|f\|_{\strut H^{-s}(\mathbb{X})}.$$
Now, to show $T$ is bounded on $\strut L^2 (\mathbb{R},|c(\lambda)|^{-2} d\lambda),$  it is  sufficient to show that there exists a constant $C>0$ such that  \be \int_{0}^{\infty} |K(\lambda,\eta)| |c(\lambda)|^{-2} d\lambda \leq C \;\; ,\forall  \eta \in (0,\infty), \label{firstestimate}\ee
and \be \int_{0}^{\infty} |K(\lambda,\eta)| |c(\eta)|^{-2} d\eta \leq C \;\; .\forall  \lambda  \in (0,\infty). \label{secondestimate}\ee
Due to the symmetry in the expression for $K(\lambda,\eta)$ , it is sufficient to prove \eqref{firstestimate}, and \eqref{secondestimate} will follow. Performing a change of variables $u=({\lambda}^2 + {\rho}^2)^{\frac{a}{2}}$, and $b=({\eta}^2 +{\rho}^2)^{\frac{a}{2}},$ using the estimate of $ c(\lambda)$ and the fact that  $\rho\geq 1$, and by boundedness of $\phi_\lambda$ the \eqref{firstestimate} becomes
\begin{align}& \int_{0}^{\infty} |K(\lambda,\eta)| |c(\lambda)|^{-2} d\lambda \nonumber  \\ &\leq c \displaystyle \int_{{\rho}^a}^{\infty} |\hat{\psi}(u-b)|\left(\int_{\mathbb{X}} \alpha_0(x)^2\phi_{\lambda}(x)\phi_{\eta}(x) dx\right) 
\big( 1+\sqrt{b^{\frac{2}{a}}-{\rho}^2}\big)^s
\big(1+\sqrt{u^{\frac{2}{a}}-{\rho}^2}\big)^{s+m_1+m_2 -1} u^{\frac{2}{a} -1} du. \label{thirdestimate}
\end{align} 
We now estimate $K(\lambda,\eta)$ into two cases:
\subsection*{Case I\;(  $\eta \leq 2$)\;:}
Given that $\eta \leq 2$, it follows that $b \leq (4 + \rho^2)^{\frac{a}{2}}$. The integral in \eqref{thirdestimate} can be expressed as 
\begin{align*}  
\int_{0}^{\infty} |K(\lambda,\eta)| |c(\lambda)|^{-2} d\lambda &\leq \int_{{\rho}^a}^{{\rho}^a + \alpha} |\hat{\psi}(u-b)|u^{\frac{s+m_1+m_2+1-a}{a}} du + \int_{{\rho}^a + \alpha}^{\infty} |\hat{\psi}(u-b)|u^{\frac{s+m_1+m_2+1-a}{a}} du\\ \nonumber
&=I_1 + I_2.
\end{align*}
As $\hat{\psi}$ is rapidly decreasing  and  bounded, it is clear that $I_1 \leq C$. Also, for large positive integer $N$
$$I_2 \leq \int_{{\rho}^a + \alpha}^{\infty} \displaystyle\frac{u^{\gamma}}{u^N} du \leq C,$$ 
where $\gamma=\frac{s+m_1+m_2+1-a}{a}$.
 
\subsection*{Case II\;(  $\eta > 2$)\;:}
Note that for $\eta > 2,$ 
$\big( 1+\sqrt{b^{\frac{2}{a}}-{\rho}^2}\big)^s \leq Cb^{\frac sa}$ and   $b>(2^2+{\rho}^2)^{\frac{a}{2 }} >5^{\frac{a}{2}}$. Now by using these estimates in  \eqref{firstestimate} we get,

\begin{align*}
 & \int_{0}^{\infty} |K(\lambda,\eta)| |c(\lambda)|^{-2} d\lambda \nonumber\\
 \leq &Cb^{\frac{s}{a}} \int_{{\rho}^a}^{\infty }\left(\int_{\mathbb{X}} \alpha_0(x)^2\phi_{\lambda}(x)\phi_{\eta}(x) dx \right) \hat{\psi}(u-b)\left(1+\sqrt{u^{\frac{2}{a}}-{\rho}^2}\right)^{s+m_1+m_2-1} u^{\frac{2}{a} -1} du \nonumber \\
  =& C b^{\frac{s}{a}} \int_{{\rho}^a}^{(3+{\rho}^2)^{\frac{a}{2}} }\left(\int_{\mathbb{X}} \alpha_0(x)^2\phi_{\lambda}(x)\phi_{\eta}(x) dx\right) \hat{\psi}(u-b) \left(1+\sqrt{u^{\frac{2}{a}}-{\rho}^2}\right)^{s+m_1+m_2-1} u^{\frac{2}{a} -1} du \nonumber \\
&+ C b^{\frac{s}{a}} \int_{(3+{\rho}^2)^{\frac{a}{2}}}^{\frac{3b}{2} }\Big(\int_{\mathbb{X}} \alpha_0(x)^2\phi_{\lambda}(x)\phi_{\eta}(x) dx\Big) \hat{\psi}(u-b)\Big(1+\sqrt{u^{\frac{2}{a}}-{\rho}^2}\Big)^{s+m_1+m_2-1} u^{\frac{2}{a} -1} du\nonumber \\
&+C b^{\frac{s}{a}} \int_{\frac{3b}{2}}^{\infty }\Big(\int_{\mathbb{X}} \alpha_0(x)^2\phi_{\lambda}(x)\phi_{\eta}(x) dx\Big) \hat{\psi}(u-b)\Big(1+\sqrt{u^{\frac{2}{a}}-{\rho}^2}\Big)^{s+m_1+m_2-1} u^{\frac{2}{a} -1} du  \\
=& I_1 + I_2 + I_3.
\end{align*}
Now we will bound the expression $I_1,\; I_2$ and $I_3$. The boundedness of $I_1$ and $I_3$ is easy. Using the decay property of $\hat{\psi}$ and the boundedness of $\phi_{\lambda},$ for large positive integer $N,$ we have 
   \be
   	I_1  \leq C b^{\frac{s}{a}} \int_{{\rho}^a}^{(3+{\rho}^2)^{\frac{a}{2}}} \frac{1}{[b-(3+{\rho}^2)^{\frac{a}{2}}]^N}
   	\leq C\displaystyle \frac{b^{\frac{s}{a}}}{b^N} 
   	 \leq C, 
   	\ee\label{I1 estimate}
   and for $\beta = \frac{s+m_1+m_2+1}{a} ,$
   \be
   	I_3  \leq Cb^{\frac{s}{a}} \int_{\frac{3b}{2}}^{\infty} \frac{u^{\beta}}{|u-b|^{N+1}} du 
    \leq Cb^{\frac{s}{a}} \int_{\frac{3b}{2}}^{\infty} \frac{u^{\beta}}{u^N} du 
    \leq C.
   		\ee
    The main and crucial step is to bound the expression $I_2$. In this case we will use use bound of the elementary spherical $\phi_{\lambda}$ in $\lambda$ variable. Since the bounds of $\phi_{\lambda}(t)$ also depends on $t$, therefore we estimate $\phi_{\lambda}(t)$ into two parts.
   	
   	In first part, we assume $t>1$. Applying the estimate \eqref{a_2estimate} and \eqref{clambda} in \eqref{Harish-Chandraseries}, we have,
    $$|\phi_{\lambda}(t)| \leq  C|c(\lambda)| \leq C \frac{1}{|\lambda|(1+|\lambda|)^{\frac{m_1+m_2-2}{2}}}.$$
    Put $\alpha=m_1+m_2-1$, then for $\lambda=(u^{\frac{2}{a}} - {\rho}^2)^{\frac{\alpha + 1}{4}}$ we have,
\be |\phi_{\lambda}(t)| \leq \frac{C}{u^{\frac{\alpha + 1}{2a}}}.\label{lambdaestimates1}\ee
Similarly , we have 
    \[|\phi_{\eta} (t)| \leq  C\frac{1}{b^{\frac{\alpha + 1}{2a}}}.\]
    
  The estimate $\phi_\lambda,$ when $t\leq 1$, is given in the following lemma.
    \begin{Lemma} For  $t\leq 1$ and $\lambda >1$, we have:
    \be \phi_{\lambda}(t) \leq \displaystyle\frac{1}{(\lambda t)^{\frac{n-1}{2}}}.\label{lambdaestimates2}\ee       
    \end{Lemma}
    \begin{proof}
     First assume $\lambda t \leq 1,$ that is   $\lambda > 1$ and, $t < \frac{1}{\lambda} < 1.$ Using the fact $|{\mathcal{J}}_{\mu}(t)| \leq C$ for all $\mu$ and for all $t \leq 1$ and the bound of $a_m$, we obtain from \eqref{spherical}:
    \[|\phi_{\lambda}(t)| \leq \frac{c}{(D(t))^{\frac{1}{2}}} \frac{1}{{\lambda}^{\frac{n-1}{2}}} \sum_{m=0}^{\infty} t^{2m} \frac{1}{R_1} < \frac{C}{(\lambda t)^{\frac{n-1}{2}}}.\]
   Next consider the case, $\lambda t \geq 1$. As $D(t)\asymp t^{n-1}$, thus by using \eqref{spherical 1}, we can write \[|\phi_{\lambda}(t)| \leq \left(\frac{t^{n-1}}{D(t)}\right)^{\frac{1}{2}} \left[{\mathcal{J}}_{\frac{n-2}{2}} (\lambda t) + E_1(\lambda t)\right] \leq C \left[{\mathcal{J}}_{\frac{n-2}{2}} (\lambda t) + E_1(\lambda t)\right].\]
   From a estimate of the Bessel function, we have $|{\mathcal{J}}_{\mu}(t)| \leq \displaystyle\frac{\Gamma(\mu + \frac{1}{2}) \Gamma (\frac 12) 2^{\mu -1}}{t^{\mu + \frac 12}},$ (for $\mu > - \frac 12$ and $t \geq 1$). Now by using the  estimate of $E_M$ given in \eqref{localestimate}, we get \[|\phi_{\lambda}(t)| \leq \frac{C}{(\lambda t)^{\frac{n-1}{2}}}.\]
       \end{proof}
      To prove the boundedness of $I_2$, we first bound the  following integral:  
      $$\displaystyle \int_{\mathbb{X}} \alpha_0(x) \phi_{\lambda}(x) \phi_{\eta}(x) dx = \int^\infty_{0} \alpha_0(s) \phi_{\lambda}(s) \phi_{\eta}(s) D(s)ds. $$
   	Now by using estimates \eqref{lambdaestimates1} and \eqref{lambdaestimates2}, we obtain,
   	  \begin{align*}
   	  	\left|\int_{\mathbb{X}} \alpha_0(x) \phi_{\lambda}(x) \phi_{\eta}(x) dx\right| & \leq \left|\int_{s>1} \alpha_0(s) \phi_{\lambda}(s) \phi_{\eta}(s) D(s)ds\right| + \left|\int_{s\leq1} \alpha_0(s) \phi_{\lambda}(s) \phi_{\eta}(s) D(s)ds\right| \\ \nonumber
   	  	& \leq \frac{C}{u^{\frac{\alpha +1}{2a}}  b^{\frac{\alpha +1}{2a}}} \left( \int_{s>1} |\alpha_0(s)|  |D(s)|ds + \int_{s \leq 1} |\alpha_0(s)| \frac{D(s)}{s^{n-1}} ds \right)\\ \nonumber
   	  	& \leq \frac{C}{u^{\frac{\alpha +1}{2a}}  b^{\frac{\alpha +1}{2a}}}.
   	  	\end{align*}
     	Finally using above bound we get,
     	\begin{align*}
     		I_2 & \leq b^{\frac{s}{a}} \int_{(3+{\rho}^2)^{\frac{a}{2}}}^{\frac{3b}{2}} |\hat{\psi}(u-b)| \;\; \frac{u^{\frac{s}{a} + \frac{\alpha}{a} + \frac{2}{a} -1}}{u^{\frac{\alpha +1}{2a}} b^{\frac{\alpha +1}{2a}}} du \\ \nonumber
     		& \leq C b^{\frac{s}{a}} \int_{(3+{\rho}^2)^{\frac{a}{2}}}^{\frac{3b}{2}} |\hat{\psi}(u-b)|\;\; \frac{u^{\frac{s}{a} + \frac{1}{a} -1}}{u^{\frac{\alpha +1}{2a}} b^{\frac{\alpha +1}{2a}}} du \\ \nonumber
     		& \leq  C b^{\frac{s}{a}} \int_{(3+{\rho}^2)^{\frac{a}{2}}}^{\frac{3b}{2}} |\hat{\psi}(u-b)| b^{\frac{s}{a} + \frac{1}{a} -1} du \\ \nonumber
     		& \leq C b^{\frac{2s}{a} + \frac{1}{a} -1} \int_{(3+{\rho}^2)^{\frac{a}{2}}}^{\frac{3b}{2}} |\hat{\psi}(u-b)| du  \leq C.
     		\end{align*}
     	
     	This proves \eqref{firstestimate}, and \eqref{firstL2estimate} follows.
      By using the definition in \eqref{mixednormdef}, the inequality in \eqref{firstL2estimate} can be written as:
      \be \|Sf\|_{\strut L^2(\strut H^0)} \leq \|f\|_{\strut H^{-s} (\mathbb{X})}.\label{Sfestimate}\ee
     
     Differentiating $Sf(x,t)$ with respect to t-variable, we have
     	\begin{align}
     	\frac{\partial}{\partial t} Sf(x,t) &= \alpha_0(x) \psi_0(t)\int_{\mathbb{X}} i({\lambda}^2 +{\rho}^2)^{\frac{a}{2}} e^{it({\lambda}^2 +{\rho}^2)^{\frac{a}{2}}} \hat{f}(\lambda) \phi_{\lambda}(x) |c(\lambda)|^{-2} d\lambda \\ \nonumber
     	& + \alpha_0(x) {\psi_0}^{\prime}(t) \int_{\mathbb{X}} e^{it({\lambda}^2 +{\rho}^2)^{\frac{a}{2}}} \hat{f}(\lambda) \phi_{\lambda}(x) |c(\lambda)|^{-2} d\lambda \\ \nonumber
     	 &= S_1f (x,t) + S_2f (x,t).
     	\end{align}
Now by equivalence of norms in sobolev spaces (see \cite[lemma 3, p 136]{ste}), we have
     	
     	\begin{align*}
     		\|Sf(x,.)\|_{\strut H^1(\mathbb{R})} &\leq C \left(\|Sf(x,.)\|_{\strut H^0(\mathbb{R})} + \|\frac{\partial}{\partial t} Sf(x,.)\|_{\strut H^0(\mathbb{R})}\right) \\ \nonumber
     		&\leq C \left( \|Sf(x,.)\|_{\strut H^0(\mathbb{R})} + \|S_1f(x,.)\|_{\strut H^0(\mathbb{R})}+\|S_2f(x,.)\|_{\strut H^0(\mathbb{R})} \right).
     		\end{align*}
 Again, by triangle inequality we have,\begin{align}\label{L2-H1estimate}
     		\|Sf\|_{\strut L^2(\strut H^1)} &\leq C \left( \| Sf\|_{\strut L^2(\strut H^0)} + \| S_1f\|_{\strut L^2(\strut H^0)} +  \| S_2f\|_{\strut L^2(\strut H^0)}\right).\\ \nonumber
     		\end{align}
Note that, in the proof of \eqref{firstL2estimate}, if we substitute $h(\eta)=\hat{f}(\eta)({\rho}^2 + \eta)^{-s}({\eta}^2 +{\rho}^2)^{\frac{a}{2}},$ then
     	\be\label{Sf1estimate}
     		\|S_1f\|_{\strut L^2(\strut H^0)}^2 \leq \|h\|_2^2 
     		\leq C \int_{\mathbb{R}} |\hat{f}(\eta)|^2 ({\eta}^2 + {\rho}^2)^{-s+a} |c(\eta)|^{-2} d\eta = \|f\|_{H^{-s+a}(\mathbb{X})}^2.
     		\ee
 and  
     	\be \|S_2f\|_{\strut L^2(\strut H^0)} \leq \|f\|_{\strut H^{-s} (\mathbb{X})}.\label{Sf2estimate}\ee
 From \eqref{L2-H1estimate}, \eqref{Sf1estimate} and \eqref{Sf2estimate}, we conclude that
     	\be \|Sf\|_{\strut L^2(\strut H^1)} \leq \|f\|_{\strut H^{-s+a}(\mathbb{X})}. \label{secondL2estimate}\ee
Interpolation between \eqref{Sfestimate} and \eqref{secondL2estimate} contributes,
     	\be \|Sf\|_{\strut L^2(\strut H^r)} \leq C\|f\|_{\strut H^{-s+ra}(\mathbb{X})},\;\;\; 0\leq r \leq 1.\ee
 By taking $r=\frac{1}{2} + \epsilon$, we have
     	$$\|Sf\|_{\strut L^2\left(\strut H^{\frac{1}{2} + \epsilon}\right)} \leq C\|f\|_{\strut H^{\frac{1}{2} + a\epsilon}(\mathbb{X})},\;\; 0\leq \epsilon \leq \frac{1}{2}.$$
  As, $\strut H^s \hookrightarrow \strut L^{\infty}$, for $s>\frac{1}{2}$, it follows that
     	$$\|Sf\|_{\strut L^2(\strut L^{\infty})} \leq \|f\|_{\strut H^s}\;\; \text{for}\; s>\frac{1}{2}.$$
      Using the above inequality and the expression of $Sf(x,t)$ given in \eqref{Sfdefinition} we got,
     	 $$\left\|\sup_{0< t< 1} |S_t f(x)|\right\|_{\strut L^2( B)} \leq \|f\|_{\strut H^s(\mathbb{X})} \;\; \text{for}\; s>\frac{1}{2}.$$
       This proves the Theorem B.

 Now we will prove the  Theorem A.
 
 If $f \in C_c^{\infty}(\mathbb{X})$ then it can be easily shown that $\lim\limits_{t\rightarrow 0} S_tf(x)=f(x)$. Now suppose $f \in \strut H^s(\mathbb{X})$ and has compact support, then there exist a sequence $\{f_n\}$ in $C_c^{\infty}(\mathbb{X})$ such that $f_n \rightarrow f$ in $\strut H^s (\X)$ and support of all $f_n$ contained in a compact set. 
      
      Applying Fatou's lemma and the estimate \eqref{maximalestimate} we get
     \be\int_{B} |S^*f(x)|^2 dx \leq \lim\limits_{n \rightarrow \infty} \int_{B}|S^*f_n(x)|^2 dx \leq \lim\limits_{n \rightarrow \infty} C\|f_n\|_{\strut H^s}^2 = C \|f\|_{\strut H^s}^2\label{lastestimate},\ee for every ball $B$ in $\mathbb{X}$.
     Now, if $f \in \strut H^s(\mathbb{X})$ and has compact support and take $g$ from $C_c^{\infty}(\mathbb{X})$ , we have
     \begin{align*}
     	\lim\limits_{t \rightarrow 0} |S_tf(x)-f(x)| &= \lim\limits_{t \rightarrow 0} |S_t(f-g)(x) - (S_tg(x)-g(x))-(f(x)-g(x))|\\ \nonumber
     	&\leq S^*(f-g)(x) + |f(x)-g(x)| .
     \end{align*}
 Taking $\strut L^2$-norm on bothside and using \eqref{lastestimate}, one obtain
 $$ \int_{B}\Big(\lim\limits_{t \rightarrow 0} |S_tf(x)-f(x)| \Big)^2 dx \leq C\|f-g\|_{\strut H^s}^2.$$
 By density, we can choose $g,$ such that $\|f-g\|_{\strut H^s}$can be made arbitrarily small, consequently 
 $$\lim\limits_{t\rightarrow 0} |S_tf(x)-f(x)|=0.$$
 This complete proof of the Theorem A.

%%%%%%%%%%%%%%%%%%%%%%%%%%%%%%%%%%%%%%%%%%%%%%%%%%%%%%%%%%%%%%%%%%%%%%%%%%%%%%%%%%%%%%%%%%%%%%%%%%%%%%%%%%%%%%%%%%%%%%%%%%%%%%%%%%%%%%%%

%%%%%%%%%%%%%%%%%%%%%%%%%%%%%%%%
 
\end{document}